\documentclass[12 pt]{article}
\usepackage{amsmath, amsfonts, amsthm, color,latexsym}
\usepackage{amsmath}
\usepackage[all]{xy}
\usepackage{amsfonts}
\usepackage{amssymb}
\usepackage{color}
\usepackage{graphicx}

\providecommand{\U}[1]{\protect\rule{.1in}{.1in}}
\newtheorem{theorem}{Theorem}[section]
\newtheorem{proposition}[theorem]{Proposition}
\newtheorem{corollary}[theorem]{Corollary}

\newtheorem{remark}[theorem]{Remark}

\newtheorem{final remark}[theorem]{Final Remark}
\newtheorem{definition}[theorem]{Definition}
\begin{document}

\title{Mid summable sequences: an anisotropic approach}
\author{Jamilson R. Campos and Joedson Santos\thanks{%
Supported by CNPq Grant 303122/2015-3\thinspace \hfill\newline
}}
\date{}
\maketitle

\begin{abstract}
The notion of mid $p$-summable sequences was introduced by Karn and Sinha in 2014 and recently explored and expanded by Botelho, Campos and Santos in 2017. In this paper we design a theory of mid summable sequences in the anisotropic setting defining a new more general space called space of mid $(q,p)$-summable sequences. As a particular case of our results, we prove an inclusion relation between spaces of  mid summable sequences. We also define classes of operators that deals with this new space, the mid $(q,p)$-summing operators, and prove some important results on these classes as inclusion and coincidence theorems and a Pietsch Domination-type theorem. It is worth to mentioning that these abovementioned results are new even in the particular case of the mid $p$-summable environment.
\end{abstract}

\section{Introduction}

\label{introd}

Classes of operators that improve the convergence of series, which are usually defined or characterized by the transformation of vector-valued sequences, are frequently studied in the theory of ideals of linear operators between Banach spaces. Certainly, the most celebrated of such classes is the ideal of absolutely $p$-summing linear operators: a linear operator $T:E \rightarrow F$ is absolutely $p$-summing if sends weakly $p$-summable sequences to absolutely $p$-summable sequences. We refer to the excellent monograph \cite{djt} for an in-depth study on this subject (see also \cite{df,pietschlivro}).

A new kind of summability in Banach spaces was introduced in 2014 by Karn and Sinha in their work \cite{karnsinha}, from a generalization of the notion of limited sets (in the sense of Gelfand-Phillips, see \cite{gelfand, phillips}) to the ``$p$-sense''. This concept led them to a new space of $E$-valued sequences, denoted $\ell_p^{\text{mid}}(E)$, such that
\begin{equation}\label{inc}
\ell_p(E) \subseteq \ell_p^{\text{mid}}(E) \subseteq \ell_p^{w}(E),
\end{equation}
where $E$ is a Banach space and $\ell_p(E)$ and $\ell_p^w(E)$ denote the spaces of absolutely and weakly $p$-summable $E$-valued sequences, respectively. They did a rudimentary study of the new space, so called space of mid $p$-summable sequences, providing some of its characteristics and studying a new class of linear operators related to this space.

In 2017, the paper of Karn and Sinha was revisited by Botelho \textit{et al.} \cite{botelhocampossantos} and much more results/features of the space of mid summable sequences was developed, including a more appropriated notation, a complete norm for $\ell_p^{\text{mid}}(E)$ and a more deep study of classes of operators that dealing with this space.

In its most appropriate notation for our purposes, the space $\ell_p^{\text{mid}}(E)$ can be defined as follows: If $p\in \lbrack 1,\infty )$, a sequence 
$(x_{j})_{j=1}^{\infty }$ in $E$ is called mid $p$-summable when
\begin{equation}
((x_{n}^{\ast }(x_{j}))_{j=1}^{\infty })_{n=1}^{\infty }\in \ell _{p}(\ell
_{p}),\ \mathrm{\ for\ all\ }(x_{n}^{\ast })_{n=1}^{\infty }\in \ell
_{p}^{w}(E^{\ast }).  \label{mid}
\end{equation}%
When equipped with the norm
\begin{equation*}
\Vert (x_{j})_{j=1}^{\infty }\Vert _{\text{mid},p}:=\sup_{(x_{n}^{\ast
})_{n=1}^{\infty }\in B_{\ell _{p}^{w}(E^{\ast })}}\left( \sum_{n=1}^{\infty
}\sum_{j=1}^{\infty }|x_{n}^{\ast }(x_{j})|^{p}\right) ^{1/p},
\end{equation*}%
$\ell _{p}^{\text{mid}}(E)$ is a Banach space.

In this paper we extend this notion to the anisotropic setting. More precisely, we define a family of more general sequence spaces, called mid  $(q,p)$-summable sequence spaces, denoted by $\ell_{q,p}^{\text{mid}}(E)$, that encompasses the space $\ell_{p}^{\text{mid}}(E)$ as a particular instance and also investigate classes of operators related to mid $(q,p)$-summable sequence spaces, as a kind of generalization of absolutely summing operators.  A lot of new results are presented such as those of inclusion and coincidence and a Pietsch-type Domination Theorem that are a novelty even in the case of mid summable sequences and operators.   

\section{Background}

We will provide the notations, definitions and results needed for our study. The letters $E,F$ shall denote Banach spaces over $\mathbb{K} = \mathbb{R}$
or $\mathbb{C}$. The closed unit ball of $E$ is denoted by $B_E$ and its
topological dual by $E^*$. The symbol $E \overset{1}{\hookrightarrow} F$
means that $E$ is a linear subspace of $F$ and $\|x\|_F \leq \|x\|_E$, for
every $x \in E$, and $E \overset{1}{=} F$ means that $E$ and $F$ are isomorphic isometrically. By $\mathcal{L}(E;F)$ we denote the Banach space of all
continuous linear operators $T \colon E \longrightarrow F$ endowed with the
usual sup norm. If $1 \le p \le \infty$ we denote by $p^*$ the conjugate of $%
p$, i.e. $1=\frac{1}{p}+\frac{1}{p^*}$. 

For a Banach space
$E$, the spaces
\[
\ell_{p}^{w}(E):=\left\{  (x_{j})_{j=1}^{\infty}\subset E:\left\Vert
(x_{j})_{j=1}^{\infty}\right\Vert _{w,p}:=\sup_{\varphi\in B_{E^{\ast}}
}\left(\sum_{j=1}^{\infty}
\left\vert \varphi(x_{j})\right\vert ^{p}\right)  ^{1/p}<\infty\right\}
\]
and
\[
\ell_{p}(E):=\left\{  (x_{j})_{j=1}^{\infty}\subset E:\left\Vert (x_{j}%
)_{j=1}^{\infty}\right\Vert _{p}:=\left(\sum_{j=1}^{\infty}
\left\Vert x_{j}\right\Vert ^{p}\right)  ^{1/p}<\infty\right\}
\]
are, respectively, the spaces of weakly and absolutely $p$-summable $E$-valued sequences. If $1\leq p\leq q<\infty,$ we say that a continuous linear operator between Banach spaces
$T:E\rightarrow F$ is $(q,p)$-summing if $\left(  T(x_{j})\right)
_{j=1}^{\infty}\in\ell_{q}(F)$ whenever $(x_{j})_{j=1}^{\infty}\in\ell
_{p}^{w}(E)$. By $\Pi_{q,p}$ we denote the ideal of absolutely $(q,p)$-summing linear operators \cite{df,djt}. If $p=q$ we simply write $\Pi_{p}$. 

An equivalent definition of absolutely summing operators asserts that $T:E\rightarrow F$ is $(q,p)$-summing
if there is a constant $C>0$ such that \[\left\Vert (T(x_{j}))_{j=1}^k\right\Vert_q \leq C\left\Vert
(x_{j})_{j=1}^{k}\right\Vert _{w,p}, \text{ for all } k \in \mathbb{N}\text{ and } x_{1},...,x_{k}\in E.\] This is also equivalent to taking $k=\infty$, whenever  $(x_{j})_{j=1}^{\infty}\in\ell_{p}^{w}(E)$. An integral characterization of absolutely summing operators was given by A. Pietsch, in his classical paper \cite{stu}: a linear operator $T$ is $p$-summing, $0<p<\infty$, if and only if there are a Borel probability measure $\mu$ on $B_{E^{\ast}}$ (with the
weak-star topology) and a constant $C>0$ such that%
\begin{equation}
\left\Vert T(x)\right\Vert \leq C\left(  \int_{B_{E^{\ast}}}\left\vert
\varphi(x)\right\vert ^{p}d\mu\right)  ^{\frac{1}{p}}, \text{ for all }x \in E. \label{p111}%
\end{equation}

We say that a normed space $E$ has cotype $q\geq2$ if there is a constant
$K\geq0$ so that, for all positive integer $n$ and all $x_{1},...,x_{n}$ in
$E$, we have%
\begin{equation}
\left( \sum_{i=1}^n\left\Vert x_{i}\right\Vert ^{q}\right)  ^{1/q}\leq K\left(  \int_{0}\left\Vert \sum_{i=1}^n r_{i}\left(  t\right)  x_{i}\right\Vert ^{2}dt\right)  ^{1/2}\text{,}
\label{2.3}
\end{equation}
where $r_{i}$'s are the Rademacher functions.

In the Section \ref{operatorsec} we will study the operator ideals determined by the transformation of vector-valued sequences belonging to the known sequence spaces and our new space. A usual approach, proving all the desired properties using the definitions of the involved operators and underlying sequence spaces, would lead to long and boring proofs. Alternatively, we shall apply the abstract framework constructed in \cite{botelhocampos} that generates this king of ideals.  So, let us present some required definitions of this environment with which we close this section.

A sequence class $X$ is a rule that assigns to each $E$ a Banach space $X(E)$ of $E$-valued sequences, that is $X(E)$ is a vector subspace of $E^{\mathbb{N}}$ with the coordinatewise operations, such that
$$c_{00}(E) \subseteq X(E) \stackrel{1}{\hookrightarrow}  \ell_\infty(E){\rm ~ and~ } \|e_j\|_{X(\mathbb{K})}= 1, {\rm ~for~ every~} j.$$
It is customary to denote a sequence class $X$ by the symbol $X(\cdot)$. A sequence class $X$ is called finitely determined if for every sequence $(x_j)_{j=1}^\infty \in E^{\mathbb{N}}$, $(x_j)_{j=1}^\infty \in X(E)$ if and only if $\displaystyle\sup_k \left\|(x_j)_{j=1}^k  \right\|_{X(E)} < +\infty$ and, in this case,
$$\left\|(x_j)_{j=1}^\infty  \right\|_{X(E)} = \sup_k \left\|(x_j)_{j=1}^k  \right\|_{X(E)}. $$

Classes of linear operators that improve convergence of series fit into the sequence classes environment as follows. A linear operator $T \in {\cal L}(E;F)$ is called $(X;Y)$-summing if $(T(x))_{j=1}^\infty \in Y(F)$ whenever $(x_j)_{j=1}^\infty \in X(E)$. In this case we write $T \in {\cal L}_{X;Y}(E;F)$ and define
$$\|T\|_{X;Y} = \|\widehat{T}\|_{{\cal L}(X(E); Y(F))},$$ 
where $\widehat{T}$ is the linear operator induced by $T$, given by $\widehat{T}((x_j)_{j=1}^\infty ) = (T(x))_{j=1}^\infty$.

Finally, we say that a sequence classe $X$ linearly stable if  ${\cal L}_{X;X}(E;F) \stackrel{1}{=} {\cal L}(E;F)$, for all Banach spaces $E$ and $F$, that is, for every $T \in {\cal L}(E;F)$, $(T(x_j))_{j=1}^\infty \in X(F)$ whenever $(x_j)_{j=1}^\infty \in X(E)$ and $\|\widehat{T}\colon X(E) \longrightarrow X(F)\| = \|T\|  $.

\section{Anisotropic mid summable sequences}

We start this section with a fact that motivates one of our main definitions. In that follows,  $1 \le p,q < \infty$ are real numbers and $E$ a Banach space. \medskip

The space $\ell_{p}^w(E^*)$ can be identified with ${\cal L}(E;\ell_{p})$, associating  $x^*=(x_{n}^*)_{n=1}^\infty \in \ell_{p}^w(E^*)$ with the operator $T_{x^*} \in {\cal L}(E;\ell_{p})$ given by $T_{x^*}(a) = (x_n^*(a))_{n=1}^\infty$ (see \cite[Section 8.2]{df}). Thus, for any sequence $(x_{j})_{j=1}^\infty \in E^{\mathbb{N}}$, we have $$T_{x^*}(x_j) = (x^*_{n}(x_{j}))_{n=1}^\infty \in \ell_p,\ \text{for all } j \in \mathbb{N},\ \text{whenever } x^*=(x_{n}^*)_{n=1}^\infty \in \ell_{p}^w(E^*).$$ 

The question that arises is: what happens when we take both indexes into account? In other words, what happens with the sequence $(T_{x^*}(x_j))_{j=1}^\infty = ((x^*_{n}(x_{j}))_{n=1}^\infty)_{j=1}^\infty$? This question lead us to the following definition:

\begin{definition}\label{esp-novo}  \rm
A sequence $(x_{j})_{j=1}^\infty \in E^{%
\mathbb{N}}$ is \textit{mid $(q,p)$-summable} if
\begin{equation}  \label{seqdef}
((x^*_{n}(x_{j}))_{n=1}^\infty)_{j=1}^\infty \in \ell_{q}(\ell_{p}),\ \
\mathrm{whenever\ \ } (x^*_{n})_{n=1}^\infty \in \ell_p^w(E^{*}),
\end{equation}
that is, 
\begin{equation}  \label{desigdef}
\sum_{j=1}^\infty \left(\sum_{n=1}^\infty |x^*_{n}(x_{j})|^{p}\right)^%
{q/p} < \infty,\ \ \mathrm{whenever\ \ } (x^*_{n})_{n=1}^\infty \in
\ell_p^w(E^{*}).
\end{equation}
\end{definition}

By the motivation above, the Definition \ref{esp-novo} is also equivalent to:\\ 

\noindent $\bullet$ A sequence $(x_{j})_{j=1}^\infty \in E^{\mathbb{N}}$ is mid $(q,p)$%
-summable if 
\begin{equation}\label{opdef}
(T(x_{j}))_{j=1}^{\infty }\in \ell_q(\ell_p),\
\mathrm{for\ all}\ T\in \mathcal{L}(E;\ell_p).
\end{equation}

\begin{remark}\rm
An attentive reader must have noticed the change of the order of indexes in the expressions \eqref{mid} and \eqref{seqdef}. Clearly, the indexes in \eqref{mid} can be interchanged and this is commonly done in the papers \cite{botelhocampossantos} and \cite{karnsinha} as a technique to obtain important results. So, this fact and our discussion at the beginning of this section naturally suggests the order of indexes in \eqref{seqdef} for our new definition. 
\end{remark}

We denote the set of the mid $(q,p)$-summable $E$-valued sequences by $\ell _{q,p}^{\text{mid}}(E)$. Of course, if $q=p$ we recover the space $\ell_{p}^{\text{mid}}(E)$.

\medskip
The expression
\begin{equation}
\Vert (x_{j})_{j=1}^{\infty }\Vert _{q,p}:=\sup_{(x_{n}^{\ast
})_{n=1}^{\infty }\in B_{\ell _{p}^{w}(E^{\ast })}}\left( \sum_{j=1}^{\infty
}\left( \sum_{n=1}^{\infty }|x_{n}^{\ast }(x_{j})|^{p}\right) ^{{q}/{%
p}}\right) ^{1/{q}}  \label{norma}
\end{equation}%
defines a norm on the space $\ell _{q,p}^{\text{mid}}(E)$. Indeed, the finiteness
of (\ref{norma}) can be proved using  that $\ell _{p}(E),\ell _{p}^{w}(E)%
\overset{1}{\hookrightarrow }\ell _{\infty }(E)$, for all $p$ and
for any Banach space $E$, and the Closed Graph Theorem. The other properties
of a norm are easily verified.
\medskip

For all Banach space $E$ and all $p$, it is immediate that \[\ell_{q,p}^{\text{mid}}(E) \overset{1}{\hookrightarrow} \ell_{r,p}^{\text{mid}}(E),\ \text{if } q \le r,\] and, similar to that found in \cite[Proposition 1.4]{botelhocampossantos}, the space $%
\ell_{q,p}^{\text{mid}}(E)$ can be placed in a chain with the spaces of
absolutely and weakly summable sequences. It is what shows the following
result.

\begin{proposition}\label{propchain}
The following chain is verified, for all $q$ and $p$:
\begin{equation}\label{chain}
\ell_{q}(E) \overset{1}{\hookrightarrow} \ell_{q,p}^{\mathrm{mid}}(E) \overset{1}{\hookrightarrow}
\ell_{q}^w(E)
\end{equation}
\end{proposition}

\begin{proof}	 
For every $x^* \in B_{E^{*}}$, it is clear that $(x^*,0,0,\ldots) \in B_{\ell_p^w(E^{*})}$, so $\|\cdot \|_{q,w} \leq \|\cdot\|_{q,p}$ in $\ell_{q,p}^{\text{mid}}(E)$, for all $q$ and $p$.  Given $(x_{j})_{j=1}^\infty \in \ell_{q}(E)$ and $(x_n^*)_{n=1}^\infty \in \ell_{p}(E^{*})$. Using that $B_E$, regarded as a subspace of $E^{**}$, is a norming subset of $E^{**}$, we have 
\begin{equation}\label{truquenorma}
\|(x_n^{*})_{n=1}^\infty\|_{p,w}^{p} = \sup\limits_{x \in B_E} \sum\limits_{n=1}^\infty |x_n^*(x)|^{p}.
\end{equation}
 So putting $J = \{j \in \mathbb{N} : x_j \neq 0\}$ we have
\begin{align*}
\left(\sum_{j=1}^\infty \left(\sum_{n=1}^\infty |x_{n}^*(x_{j})|^{p}\right)^{{q}/{p}}\right)^{1/{q}} &  = \left(\sum_{j \in J}\|x_j\|^{q} \cdot \left( \sum_{n=1}^\infty \left|x_n^*\left(\frac{x_j}{\|x_j\|}\right)\right|^{p}\right)^{q/p}\right)^{1/q} \\
& \leq \| (x_{j})_{j=1}^\infty\| _{q} \| (x_{n}^*)_{n=1}^\infty\|_{p,w},
\end{align*}
from which $\|\cdot\|_{q,p} \le \|\cdot\|_{q}$ in $\ell_{q}(E)$.
\end{proof}

\begin{remark}\label{inc00}\rm
As $c_{00}(E) \subseteq \ell_q(E)$ and $\ell_{q}^w(E) \overset{1}{\hookrightarrow} \ell_\infty(E)$, using the chain in \eqref{chain} we obtain that $$c_{00}(E) \subseteq  \ell_{q,p}^{\text{mid}}(E) \overset{1}{\hookrightarrow} \ell_\infty(E),\ \text{for all }q\ \text{and }p.$$ 
\end{remark}

Once known the place of the space $\ell_{q,p}^{\text{mid}}(E)$ between the
spaces of absolutely and weakly summable sequences, it becomes a natural
question to know the conditions for which there is a coincidence with these
spaces. The following proposition establishes these conditions and its proof, which will be omitted, stems from adaptations in the demonstration of the facts from \cite[Proposition 3.1 and Theorem 4.5]{karnsinha}. The reciprocal of the sentence (ii) below is valid if $q=p$.

\begin{proposition}
Let $E$ be a Banach space and $1 \leq p < \infty$. Then:
\begin{description}
\item[(i)] $\ell_{q,p}^{\mathrm{mid}}(E) = \ell_{q}^w(E)$ if and only if  $\Pi_{q}(E;\ell_{p}) = \mathcal{L}(E;\ell_{p})$.
\item[(ii)] If $q \le p$ and $\ell_{q,p}^{\mathrm{mid}}(E) = \ell_{q}(E)$ then $E$ is a subspace of $L_q(\mu)$, for some Borel measure $\mu$.
\end{description}
\end{proposition}

Before we show the next theorem, the following remark presents some required facts.

\begin{remark}\label{obs2}\rm
Let $x = (x_{j})_{j=1}^\infty \in \ell_{q}^w(E)$. It
is immediate that the operator
\begin{equation*}
\Psi_x :E^{*} \rightarrow \ell_{q};\ \Psi_x(x^*) :=
(x^*(x_j))_{j=1}^\infty,
\end{equation*}
is well-defined, linear and continuous, with $\|\Psi_x\| =
\|(x_{j})_{j=1}^\infty\|_{q,w}$. Furthermore, if $q \le p$, thanks to Minkowski's inequality, we
have
\begin{equation}  \label{mink}
\left(\sum_{n=1}^\infty \left(\sum_{j=1}^\infty
|x_{n}^*(x_{j})|^{q}\right)^{{p}/{q}}\right)^{1/{p}}
\leq\left(\sum_{j=1}^\infty \left(\sum_{n=1}^\infty
|x_{n}^*(x_{j})|^{p}\right)^{{q}/{p}}\right)^{1/{q}}
\end{equation}
and if $x = (x_{j})_{j=1}^\infty \in \ell_{q,p}^{\text{mid}}(E)$, then the
induced operator
\begin{equation*}
\widetilde{\Psi_x} :\ell_p^w(E^{*}) \rightarrow \ell_{p}(\ell_{q});\
\widetilde{\Psi_x}((x_n^*)_{n=1}^\infty) :=
((x_n^*(x_j))_{j=1}^\infty)_{n=1}^\infty
\end{equation*}
is well-defined, linear and, by the Closed Graph Theorem (see \cite%
{botelhocampossantos}, Proposition 1.4), is continuous with $\|\widetilde{%
\Psi_x}\| \le \|(x_{j})_{j=1}^\infty\|_{q,p}$ (the equality occurs
when $q=p$). 
\end{remark}

The following theorem provides an result on mid $(q,p)$-summable sequences in terms of absolutely summing operators that, besides its intrinsic relevance, will be very useful to prove other important results throughout our work. 

\begin{theorem}
\label{caracter} Let $x = (x_{j})_{j=1}^\infty \in \ell_{q}^w(E)$. Consider the following sentences:
\begin{description}
\item[(i)] $x = (x_{j})_{j=1}^\infty \in \ell_{q,p}^{\mathrm{mid}}(E)$.
\item[(ii)] $\Psi_x \in \Pi_{p}(E^{*};\ell_{q})$.
\end{description}
So occurs that:
\begin{itemize}
\item If $q \le p$, then $\mathrm{(i)} \Rightarrow \mathrm{(ii)}$ and $%
\pi_{p}(\Psi_x) \le \|(x_{j})_{j=1}^\infty\|_{q,p}$.
\item If $q \ge p$, then $\mathrm{(ii)} \Rightarrow \mathrm{(i)}$ and $%
\|(x_{j})_{j=1}^\infty\|_{q,p} \le \pi_{p}(\Psi_x)$.
\end{itemize}
Of course, the sentences are equivalent and the equality of norms holds when
$q=p$.
\end{theorem}

\begin{proof}
\noindent $\mathrm{(i)} \Rightarrow \mathrm{(ii)}$: By Remark \ref{obs2}, if $x = (x_{j})_{j=1}^\infty \in \ell_{q,p}^{\text{mid}}(E)$ and since $q\le p$, we have
\begin{align*}
\left(\sum_{n=1}^{\infty}\|(\Psi_x(x^*_n))_{n=1}^\infty\|_{q}^{p}\right)^{1/{p}} & = \left(\sum_{n=1}^\infty \left(\sum_{j=1}^\infty |x_{n}^*(x_{j})|^{q}\right)^{{p}/{q}}\right)^{1/{p}} \\
& \le \left(\sum_{j=1}^\infty \left(\sum_{n=1}^\infty |x_{n}^*(x_{j})|^{p}\right)^{{q}/{p}}\right)^{1/{q}} < \infty,
\end{align*}
whenever $(x_n^*)_{n=1}^\infty \in \ell_p^w(E^*)$. So, $\Psi_x \in \Pi_{p}(E^{*};\ell_{q})$ and $\pi_{p}(\Psi_x) = \|\widetilde{\Psi_x}\| \le \|(x_{j})_{j=1}^\infty\|_{q,p}$

\noindent The implication $\mathrm{(ii)} \Rightarrow \mathrm{(i)}$ follows by the hypothesis and an inversion of the above inequality, since $q \ge p$.
\end{proof}

It is known that $\Pi_{p}(E;F) \overset{1}{\hookrightarrow} \Pi_{s}(E;F)$, if $p \leq s$ (see \cite{djt}, Theorem 10.4). Joining this inclusion result and the Theorem \ref{caracter} we obtain the next proposition.

\begin{proposition}
	\label{inclus} If $(q,p)$ and $(r,s)$ are parameters such that $q \le p \le s \le r$. Then, for every Banach space $E$, we have
	\begin{equation}  \label{inc2}
	\ell_{q,p}^{\mathrm{mid}}(E) \overset{1}{\hookrightarrow}
	\ell_{r,s}^{\mathrm{mid}}(E).
	\end{equation}
	In particular,
	\begin{equation}  \label{incmid}
	\ell_p^{\mathrm{mid}}(E) \overset{1}{\hookrightarrow} \ell_s^{\mathrm{mid}}(E),\ \mathrm{if\ }
	p\leq s.
	\end{equation}
\end{proposition}

\begin{proof}
	Under these assumptions we have
	\begin{align*}
	x = (x_{j})_{j=1}^\infty \in \ell_{q,p}^{\text{mid}}(E) & \Rightarrow \Psi_x \in \Pi_{p}(E^{*};\ell_{q})\\
	& \Rightarrow \Psi_x \in \Pi_{s}(E^{*};\ell_{q}) \\
	& \Rightarrow \Psi_x \in \Pi_{s}(E^{*};\ell_{r}) \Rightarrow x = (x_{j})_{j=1}^\infty \in \ell_{r,s}^{\text{mid}}(E).
	\end{align*}
	We also obtain, from above calculus, that
	\[\|(x_{j})_{j=1}^\infty\|_{r,s} \le \pi_{s}(\Psi_x) \le \pi_{p}(\Psi_x) \le \|(x_{j})_{j=1}^\infty\|_{q,p}.\]
\end{proof}

\begin{remark}\rm
The inclusion in (\ref{incmid}) has not been established in the
paper \cite{karnsinha} neither in the paper \cite{botelhocampossantos} and
as far as we know it had not yet been proven in the literature.
\end{remark}

We finish this section showing that the space of mid $(q,p)$-summable sequences fits in the abstract framework of sequence classes. This result has great importance in the next section.

\begin{proposition}
\label{seqclass} The correspondence $E \mapsto \ell_{q,p}^{\mathrm{mid}}(E)$ is a finitely determined and linearly stable sequence class.
\end{proposition}

\begin{proof}
As $\|e_j\|_{q,w} = \|e_j\|_{q} =1$, the Proposition \ref{propchain} ensures that $\|e_j\|_{q,p} =1$  in $\ell_{q,p}^{\text{mid}}(\mathbb{K})$, where $e_j$ is the $j$-th canonical unit scalar-valued sequence. The Remark \ref{inc00} gives us the remaining properties to prove that $\ell_{q,p}^{\text{mid}}(\cdot)$ is a sequence class.

Let $(x_j)_{j=1}^\infty$ be a $E$-valued sequence. If $(x_j)_{j=1}^\infty \in \ell_{q,p}^{\text{mid}}(E)$, then the double series in (\ref{desigdef}) converges and, therefore,
$$\sup_{(x_n^*)_{n=1}^\infty \in  B_{\ell_p^w(E^{*})}}\sum_{j=1}^\infty \left(\sum_{n=1}^\infty |x_n^*(x_j)|^{p}\right)^{q/p} = \sup_k \sup_{(x_n^*)_{n=1}^\infty \in  B_{\ell_p^w(E^{*})}}\sum_{j=1}^k \left(\sum_{n=1}^\infty |x_n^*(x_j)|^{p}\right)^{q/p}.$$
Conversely, if the sup in the right side of above equality is finite then $(x_j)_{j=1}^\infty \in \ell_{q,p}^{\text{mid}}(E)$. So,
$$(x_j)_{j=1}^\infty \in \ell_{q,p}^{\text{mid}}(E)\ \mathrm{if\ and\ only\ if}\ \sup\limits_k \|(x_j)_{j=1}^k\|_{q,p} < \infty$$ and
$\|(x_j)_{j=1}^\infty\|_{q,p} = \sup\limits_k \|(x_j)_{j=1}^k\|_{q,p}$. Thus, $\ell_{q,p}^{\text{mid}}(\cdot)$ is a finitely determined. 

To prove the linear stability of $\ell_{q,p}^{\text{mid}}(\cdot)$, let $T \in {\cal L}(E;F)$ and $(x_j)_{j=1}^\infty \in \ell_{q,p}^{\text{mid}}(E)$. By the linear stability of $\ell_p^{w}(\cdot)$  \cite[Theorem 3.3]{botelhocampos}, we have $(T^*(y_n^*))_{n=1}^\infty = (y_n^* \circ T)_{n=1}^\infty\in \ell_p^w(E^{*})$, for every $(y_n^*)_{n=1}^\infty \in \ell_p^w(F^{*})$, where $T^{*}\colon F^{*} \longrightarrow E^{*}$ is the adjoint of $T$. Therefore,
\[\left(\left(y_n^*(T\left(x_j\right))\right)_{j=1}^\infty\right)_{n=1}^\infty = \left(\left(y_n^* \circ T\left(x_j\right)\right)_{j=1}^\infty\right)_{n=1}^\infty \in \ell_{q}\left(\ell_{p}\right)\]
and $(T(x_j))_{j=1}^\infty \in \ell_{q,p}^{\text{mid}}(F)$, whenever $(x_j)_{j=1}^\infty \in \ell_{q,p}^{\text{mid}}(E)$. Defining the induced operator $\widehat{T}\colon \ell_{q,p}^{\text{mid}}(E) \longrightarrow \ell_{q,p}^{\text{mid}}(F)$ by $\widehat{T}((x_j)_{j=1}^\infty) = (T(x_j))_{j=1}^\infty$, a standard calculation shows that $\|T\|= \|\widehat{T}\|$.
\end{proof}

\section{Mid summing operators}\label{operatorsec}

Now we will study classes of operators that
are characterized by transformations of sequences involving the space of mid
$(q,p)$-summable sequences.\medskip

From now on $1 \leq q, p, r <\infty$ are real numbers and $T \in
\mathcal{L}(E;F)$ is a continuous linear operator.

\begin{definition}\label{defopA} \rm
If $q\leq r$, an operator $T$
is said to be \emph{absolutely mid $(r;q,p)$-summing} if
\begin{equation}
\left( T\left( x_{j}\right) \right) _{j=1}^{\infty }\in \ell _{r}(F)\ \
\mathrm{whenever}\ \ (x_{j})_{j=1}^{\infty }\in \ell
_{q,p}^{\text{mid}}(E).  \label{aps}
\end{equation}%
\end{definition}

\begin{definition}\label{defopF} \rm
If $q \geq r$. An operator $T$
is said to be \emph{weakly mid $(q,p;r)$-summing} if
\begin{equation}
\left( T\left( x_{j}\right) \right) _{j=1}^{\infty }\in \ell
_{q,p}^{\text{mid}}(F)\ \ \mathrm{whenever}\ \ (x_{j})_{j=1}^{\infty }\in
\ell _{r}^{w}(E).  \label{wps}
\end{equation}
\end{definition}
We denote respectively by $\Pi_{r;q,p}^{\text{mid}}(E;F)$ and $W_{q,p;r}^{\text{mid}}(E;F)$ the above spaces. 

\medskip A standard calculation shows that if $r < q$ then $
\Pi_{r;q,p}^{\text{mid}}(E;F) = \{0\}$ and if $r > q$ then $
W_{q,p;r}^{\text{mid}}(E;F) = \{0\}$. This facts justify the constrains over $q$ and $r$ in Definitions \ref{defopA} and  \ref{defopF}.

\medskip The above definitions recover the classes of mid summing operators
as defined in the paper \cite[Definition 2.1]{botelhocampossantos}. Indeed, 
\begin{equation*}
\Pi_{q;p}^{\text{mid}}(E;F) = \Pi_{q;p,p}^{\text{mid}}(E;F)\ \mathrm{and\ }
W_{q;p}^{\text{mid}}(E;F) = W_{q,q;p}^{\text{mid}}(E;F).
\end{equation*}
Moreover, when $q=p$ we recover the classes
\begin{equation}\label{notamid}
\Pi_{p}^{\text{mid}}(E;F) = \Pi_{p;p,p}^{\text{mid}}(E;F)\ \mathrm{and\ }
W_{p}^{\text{mid}}(E;F) = W_{p,p;p}^{\text{mid}}(E;F).
\end{equation}

\begin{remark}\rm 
Follows immediately from above definitions that if $1 \le r,s < \infty$ are real numbers such that $r \le q \le s$, then
\begin{equation*}
\Pi_{s;q,p}^{\text{mid}} \circ W_{q,p;r}^{\text{mid}} \subseteq \Pi_{s,r}.
\end{equation*}
\end{remark}

Using the facts proved in Proposition \ref{seqclass} and the abstract approach in \cite{botelhocampos}, the next two theorems
are immediate consequences of \cite[Proposition 1.4]{botelhocampos}, with
the equivalences involving $\ell_p^u(E)$, in Theorem \ref{teoweak}, due to
\cite[Corollary 1.6]{botelhocampos}.

\begin{theorem}
	The following sentences are equivalent:
	
	\begin{description}
		\item[(i)] $T\in \Pi_{r;q,p}^{\mathrm{mid}}(E;F)$.
		
		\item[(ii)] The induced map $\widehat{T}\colon \ell_{q,p}^{\mathrm{mid}}(E)
		\longrightarrow \ell_r(F)$ is a well defined continuous linear operator.
		
		\item[(iii)] There is a constant $A> 0$ such that $\left\|\left(T\left(x_j%
		\right)\right)_{j=1}^k\right\|_r \leq A
		\left\|(x_j)_{j=1}^k\right\|_{q,p}$ for every $k \in \mathbb{N}$ and
		all $x_j \in E$, $j=1,\ldots,k$.
		
		\item[(iv)] There is a constant $A > 0$ such that $\left\|\left(T\left(x_j%
		\right)\right)_{j=1}^\infty\right\|_r \leq A
		\left\|(x_j)_{j=1}^\infty\right\|_{q,p}$ for every $%
		\left(x_j\right)_{j=1}^\infty \in \ell_{q,p}^{\mathrm{mid}}(E)$.
	\end{description}
	
	Furthermore,
	\begin{equation*}
	\|T\|_{\Pi_{r;q,p}^{\mathrm{mid}}}:= \|\widetilde{T}\| = \inf\{A: \mathrm{%
		(iii)~holds}\} = \inf\{A: \mathrm{(iv)~holds}\}.
	\end{equation*}
\end{theorem}

\begin{theorem}
	\label{teoweak} The following sentences are equivalent:
	
	\begin{description}
		\item[(i)] $T\in W_{q,p;r}^{\mathrm{mid}}(E;F)$.
		
		\item[(ii)] The induced map $\widehat{T}\colon \ell_r^{w}(E)
		\longrightarrow \ell_{q,p}^{\mathrm{mid}}(F)$ is a well defined continuous
		linear operator.
		
		\item[(iii)] $\left(T\left(x_j\right)\right)_{j=1}^\infty \in
		\ell_{q,p}^{\mathrm{mid}}(F)$ whenever $(x_j)_{j=1}^\infty \in \ell_{r}^{u}(E)$.
		
		\item[(iv)] The induced map $\widetilde{T}\colon \ell_r^{u}(E) \longrightarrow
		\ell_{q,p}^{\mathrm{mid}}(F)$ is a well defined continuous linear operator.
		
		\item[(v)] There is a constant $B >0$ such that $\left\|\left(T\left(x_j%
		\right)\right)_{j=1}^k\right\|_{q,p} \leq B
		\left\|(x_j)_{j=1}^k\right\|_{w,r}$ for every $k \in \mathbb{N}$ and all $%
		x_j \in E$, $j=1,\ldots,k$.
		
		\item[(vi)] There is a constant $B >0$ such that
		\begin{equation*}
		\left(\sum\limits_{j=1}^k \left(\sum\limits_{n=1}^\infty
		\left|y_n^*\left(T\left(x_j\right)\right)\right|^{p}\right)^{q/p}%
		\right)^{1/q}\leq B
		\left\|(x_j)_{j=1}^k\right\|_{w,r}\cdot\left\|(y_n^*)_{n=1}^\infty
		\right\|_{w,p}
		\end{equation*}
		for every $k \in \mathbb{N}$, all $x_j \in E$, $j=1,\ldots,k$, and every $%
		(y_n^*)_{n=1}^\infty \in \ell_{p}^{w}(F^*)$.
		
		\item[(vii)] There is a constant $B >0$ such that
		\begin{equation*}
		\left(\sum\limits_{j=1}^\infty \left(\sum\limits_{n=1}^\infty
		\left|y_n^*\left(T\left(x_j\right)\right)\right|^{p}\right)^{q/p}%
		\right)^{1/q}\leq B
		\left\|(x_j)_{j=1}^\infty\right\|_{w,r}\cdot\left\|(y_n^*)_{n=1}^\infty%
		\right\|_{w,p}
		\end{equation*}
		for all $\left(x_j\right)_{j=1}^\infty \in \ell_{r}^{w}(E)$ and $%
		(y_n^*)_{n=1}^\infty \in \ell_{p}^{w}(F^*)$.
	\end{description}
	
	Moreover,
	\begin{align*}
	\|T\|_{W_{q,p;r}^{\mathrm{mid}}}:= \|\widetilde{T}\| = \|\widehat{T}\| & =
	\inf\{B: \mathrm{(v)~holds}\} \\
	& = \inf\{B: \mathrm{(vi)~holds}\}= \inf\{B: \mathrm{(vii)~holds}\}.
	\end{align*}
\end{theorem}

\medskip
The above classes of mid summing operators enjoy the operator ideal structure:

\begin{theorem}
	\label{teoideal}The classes $\left(\Pi_{r;q,p}^{\mathrm{mid}},
	\|\cdot\|_{\Pi_{r;q,p}^{\mathrm{mid}}}\right)$ and $%
	\left(W_{q,p;r}^{\mathrm{mid}}, \|\cdot\|_{W_{q,p;r}^{\mathrm{mid}}}\right)$ are
	Banach operator ideals.
\end{theorem}

\begin{proof} Using the abstract framework, notation and language of \cite{botelhocampos}, a linear operator $T$ is absolutely mid $(r;q,p)$-summing if and only if $T$ is $\left(\ell_{q,p}^{\text{mid}}(\cdot);\ell_r(\cdot)\right)$-summing. Since $r \ge q$ we obtain $\ell_{q,p}^{\text{mid}}(\mathbb{K}) = \ell_{q}  \stackrel{1}{\hookrightarrow} \ell_r = \ell_r(\mathbb{K})$. In addition, all involved sequence classes are linearly stable. So, from \cite[Theorem 2.6]{botelhocampos} it follows that $\Pi_{r;q,p}^{\text{mid}}$ is a Banach operator ideal. The case of $W_{q,p;r}^{\text{mid}}$ can be proved similarly.
\end{proof}

\section{Inclusion and coincidence results}

In this section we are interested in the study of inclusions and coincidences  between certain classes of operators and the our classes of mid summing operators as these kinds of relations among themselves.

The first effort in this direction is given by the next theorem which asserts that $\left(\Pi_{p}\right)^{\mathrm{dual}} \subseteq W_{s,p;s}^{\text{mid}}$, for all $s \in [p,\infty)$.

\begin{theorem}
	\label{adjunto} Let $T \in \mathcal{L}(E;F)$. If the adjoint $T^*$ is
	absolutely $p$-summing, then $T \in W_{s,p;s}^{\mathrm{mid}}$, for all $s \in
	[p,\infty)$, and $\|T\|_{W_{s,p;s}^{\mathrm{mid}}} \le \pi_{p}(T^*)$.
\end{theorem}

\begin{proof}
	For all $k,l \in \mathbb{N}$ we can calculate
	\begin{align}\label{ref1}
	\left(\sum\limits_{j=1}^k \left(\sum\limits_{n=1}^l \left|y_n^*\left(T\left(x_j\right)\right)\right|^{p}\right)^{s/p}\right)^{1/s} & = \left(\sum\limits_{j=1}^k \left(\sum\limits_{n=1}^l \left| T^*(y_n^*)(x_j)\right|^{p}\right)^{s/p}\right)^{1/s} \nonumber \\
	& \le \left(\sum\limits_{n=1}^l \left(\sum\limits_{j=1}^k \left| T^*(y_n^*)(x_j)\right|^{s}\right)^{p/s}\right)^{1/p} \nonumber \\
	& = \left(\sum\limits_{n=1}^l \|T^*(y_n^*)\|^p\left(\sum\limits_{j=1}^k \left\vert \frac{T^*(y_n^*)}{\|T^*(y_n^*)\|}(x_j)\right\vert^s \right)^{p/s}\right)^{1/p}  \\
	& \le  \left\Vert (T^*(y_n^*))_{n=1}^l\right\Vert_p \left\Vert (x_j)_{j=1}^k\right\Vert_{w,s} \nonumber\\
	& \le \pi_{p}(T^*) \left\Vert (x_j)_{j=1}^k\right\Vert_{w,s} \left\Vert (y_n^*)_{n=1}^l\right\Vert_{w,p}. \nonumber
	\end{align}
	So, if $(x_j)_{j=1}^\infty \in \ell_s^w(E)$ and $(y_n^*)_{n=1}^\infty \in \ell_p^w(F^*)$, taking $k,l \rightarrow \infty$ in (\ref{ref1}), we get what we want to prove.
\end{proof}

The next relationship between the classes of Cohen strongly summing and
weakly mid summing linear operators is obtained as a immediate consequence of Theorem \ref{adjunto}.

\begin{corollary}
	If $1=\frac{1}{p}+ \frac{1}{p^*}$, then $\mathcal{D}_{p^*} \subseteq
	W_{p}^{\mathrm{mid}}$.
\end{corollary}

\begin{proof}
	We know (by \cite{cohen73}) that $\left(\Pi_{p}\right)^{\mathrm{dual}} = {\cal D}_{p^*}$. So, just take $s=p$ in the Theorem \ref{adjunto}.
\end{proof}\medskip

We can also characterize the class of all weakly mid $(q,p;r)$-summing linear operators by means of multiple summing operators.

\begin{definition}\rm \cite[Definition 5.1]{gustavo} 
\label{multiplo} A $n$-linear operator $T \in \mathcal{L}%
(E_1,\ldots,E_n;F)$ is multiple $(q_1,\ldots,q_n;p_1,\ldots,p_n)$-summing if
there exists a constant $C>0$ such that
\begin{equation*}
\left(\sum_{j_n=1}^\infty \left( \cdots \left(\sum_{j_1=1}^{\infty}
\left\Vert T\left(x_{j_1}^{(1)},\ldots,
x_{j_n}^{(n)}\right)\right\Vert^{q_1} \right)^\frac{q_2}{q_1}\cdots
\right)^\frac{q_n}{q_{n-1}}\right)^\frac{1}{q_n} \le C \prod_{i=1}^{n}
\left\Vert \left( x_{j_i}^{(i)}\right)_{j_i=1}^\infty \right\Vert_{w,p_i},
\end{equation*}
for all $\left( x_{j_i}^{(i)}\right)_{j_i=1}^\infty \in \ell_{p_i}^w(E_i)$, $%
i=1,\ldots,n$. 
\end{definition}

We denote the class of all multiple $(q_1,\ldots,q_n;p_1,\ldots,p_n)$-summing $n$-linear operators by $\Pi^n_{q_1,\ldots,q_n;p_1,\ldots,p_n}(E_1,\ldots, E_n;F)$.\medskip

Given a linear operator $T \in \mathcal{L}(E;F)$, consider the bilinear
operator
\begin{equation*}
\Phi_T: F^* \times E \rightarrow \mathbb{K}\ \ \mathrm{given\ by}\ \
\Phi_T(y^*,x) = y^*(T(x)).
\end{equation*}
Just observing that
\begin{equation*}
\left(\sum\limits_{j=1}^\infty \left(\sum\limits_{n=1}^\infty
\left|y_n^*\left(T\left(x_j\right)\right)\right|^{p}\right)^{q/p}%
\right)^{1/q} = \left(\sum\limits_{j=1}^\infty
\left(\sum\limits_{n=1}^\infty
|\Phi_T(y_n^*,x_j)|^{p}\right)^{q/p}\right)^{1/q},
\end{equation*}
the Definition \ref{multiplo} together with Theorem \ref{teoweak} (vii) tell us that:

\begin{theorem}\label{caracmult}
A linear operator $T \in \mathcal{L}(E;F)$ is weakly mid $(q,p;r)$-summing if
and only if the bilinear operator $\Phi_T$ is multiple $(p,q;p,r)$-summing.
\end{theorem}

As an immediate consequence of above theorem we have the following:

\begin{corollary}\label{coinc}
Let $E$ and $F$ be Banach spaces. For an admissible choice of parameters $q,q_1,p,p_1,r$ and $r_1$:
\begin{description}
	\item[a)] If  $\Pi^{2}_{p,q;p,r}(E,F^*;\mathbb{K})=\mathcal{L}(E,F^*;\mathbb{K})$, then $W_{q,p;r}^{\mathrm{mid}}(E;F)=\mathcal{L}(E;F)$.
	\item[b)] If  $\Pi^{2}_{p,q;p,r}(E,F^*;\mathbb{K})\subseteq\Pi^{2}_{p_1,q_1;p_1,r_1}(E,F^*;\mathbb{K})$, then  $W_{q,p;r}^{\mathrm{mid}}(E;F)\subseteq W_{q_1,p_1;r_1}^{\mathrm{mid}}(E;F)$.
\end{description}	
\end{corollary}\medskip

From the above result it is expected that the research on multiple summing operators provide a variety of inclusion and coincidence theorems for the classes of weakly mid $(q,p;r)$-summing linear operators. We present some of these results in the next corollary whose proof lies on the inclusions and coincidence results of the multiple summing operators together with our Corollary \ref{coinc}. In various items is necessary remember the notations given by \eqref{notamid}.

\begin{corollary}\label{inclusandcoin}  
Let $E,F$ be Banach spaces.
\begin{description}
	\item[(i)] If $1\leq p\leq q<2$, then $W^{\mathrm{mid}}_{p}(E;F)\subset W^{\mathrm{mid}}_{q}(E;F).$
		
	\item[(ii)]  If $1\leq p,q<2$ and if $E$ and $F^*$ have cotype $2$, then
		$W^{\mathrm{mid}}_{p}(E;F)=W^{\mathrm{mid}}_{q}(E;F)$.
		
	\item[(iii)] If $1\leq p\leq q\leq\infty$ and if $E$ and  $F^*$ are $\mathcal{L}_{\infty}$-spaces, then $W^{\mathrm{mid}}_{p}(E;F)\subset W^{\mathrm{mid}}_{q}(E;F)$.
		
	\item[(iv)] If $1\leq p\leq2,$ then $W^{\mathrm{mid}}_{p}(\ell_{1};c_0)=\mathcal{L}(\ell_{1};c_0).$ 
		
\end{description}	
\end{corollary}

\begin{proof}
	The proof of (i) follows from \cite[Theorem 3.4]{davidstudia}, (ii) and (iii) are consequences of Theorems 4.6 and 4.9 in \cite{michels} and (iv) stems from \cite[Theorem 3.4]{jmaa}.
\end{proof}

The next proposition establishes a more general inclusion relation between classes of weakly mid $(q,p;r)$-summing operators, with a certain freedom of parameters not presented in previous results. Here we use again the inclusion theorem for absolutely summing operators \cite[Theorem 10.4]{djt}.

\begin{proposition}
\label{chave} If $1 \leq p, q, r, s, t<\infty$, with $r \le q$, $s \le t$ and $\dfrac{1}{r}-\dfrac{1}{q}%
\leq \dfrac{1}{s}-\dfrac{1}{t}$, then
\begin{equation*}
W_{q,p;r}^{\mathrm{mid}} \overset{1}{\hookrightarrow} W_{t,p;s}^{\mathrm{mid}}.
\end{equation*}
\end{proposition}

\begin{proof}
Let $y^* = \left( y^*_{n}\right)_{n=1}^{\infty}\in\ell_{p}^{w}(F^{*})$, $T \in W_{q,p;r}^{\text{mid}}(E;F)$ and consider the operator
\[ R_{y^*,T}:E \rightarrow \ell_{p}\ \ \mathrm{given\ by\ \ } R_{y^*,T}(x)=\left( y_n^{*}(T(x)) \right)_{n=1}^{\infty}.\]
By (\ref{truquenorma}) we obtain that $R_{y^*,T}$ is well defined and, of course, this operator belongs to  $\mathcal{L}(E;\ell_{p})$, with $\|R_{y^*,T}\|\le \|T\| \|y^*\|_{w,p}$.
Since $T$ is weakly mid $(q,p;r)$-summing, the Theorem \ref{teoweak} asserts that $R_{y^*,T}$ is absolutely $\left(q;r\right)$-summing and, by the hypothesis over our parameters, it is absolutely $(t,s)$-summing. We also have

\begin{equation*}
\pi _{(t,s)}(R_{y^*,T})\leq \pi _{(q;r)}(R_{y^*,T})\leq \|T\|_{W_{q,p;r}^{\text{mid}}}\left\Vert \left(y^*_{n}\right)_{n=1}^{\infty }\right\Vert _{w,p}
\end{equation*}
and thus	
\begin{align*}
\left(\sum_{j=1}^\infty \left(\sum_{n=1}^\infty |y_{n}^*(T(x_{j}))|^{p}\right)^{{t}/{p}}\right)^{1/{t}} & \leq \pi
_{(t;s)}(R_{y^*,T})\left\Vert \left( x_{j}\right)
_{j=1}^{\infty }\right\Vert _{w,s}\\
& \le \|T\|_{W_{q,p;r}^{\text{mid}}}\left\Vert \left(x_{j}\right) _{j=1}^{\infty }\right\Vert_{w,s} \left\Vert\left( y_{n}^*\right)
_{n=1}^{\infty }\right\Vert_{w,p},
\end{align*}
concluding our proof.
\end{proof}

We finish this paper with two results: a Pietsch Domination-type Theorem for weakly mid summing operators, obtained from Theorem \ref{caracter}, and a corollary that gives us more accurate inclusion information. The first, an important result in itself, presents an integral characterization of weakly mid summing operators and the second uses this representation to settle definitely the inclusion issue partially solved in Corollary \ref{inclusandcoin} (i).

\begin{theorem}\label{teopietsch}
Let $T \in \mathcal{L}(E;F)$. The following statements are equivalent:

\begin{description}
\item[a)] $T$ is weakly mid $p$-summing.

\item[b)] There are a constant $C>0$ and
a regular Borel probability measure $\mu$ on $B_{F^{^{** }}}$
such that
\begin{equation}  \label{777}
\left( \sum_{j=1}^{\infty}\vert x^*(T(x_j)) \vert ^{p}\right) ^{\frac{1 }{p}%
} \leq C \cdot \left(\int_{B_{F^{^{** }}}} \vert \varphi(x^*)
\vert^{p} d\mu(\varphi) \right)^{\frac{1}{p}},
\end{equation}
for all $x^*\in B_{F^{^{* }}}$ and all $(x_{j})_{j=1}^\infty \in \ell_p^w(E)$.

\item[c)] There are a constant $C>0$ and
a regular Borel probability measure $\mu$ on $B_{F^{^{** }}}$
such that, for every $k\in \mathbb{N}$,
\begin{equation}  \label{778}
\left( \sum_{j=1}^{k}\vert x^*(T(x_j)) \vert ^{p}\right) ^{\frac{1 }{p}%
} \leq C \cdot \left(\int_{B_{F^{^{** }}}} \vert \varphi(x^*)
\vert^{p} d\mu(\varphi) \right)^{\frac{1}{p}},
\end{equation}
for all $x^*\in B_{F^{^{* }}}$ and all $x_{1},..., x_{k} \in E$.
\end{description}
\end{theorem}

\begin{proof} a) $\Rightarrow$ b) If $(x_{j})_{j=1}^\infty \in
\ell_p^w(E)$, we have $y=(T(x_{j}))_{j=1}^\infty \in \ell_p^{\text{mid}}(F)$ and Theorem \ref{caracter} ensures that $\Psi_y:F^{*}\rightarrow \ell_{p}$ is $p$-summing. Therefore, there are a constant $C>0$
and a regular Borel probability measure $\mu$ on $B_{F^{**}}$
such that
\begin{equation*}
\left\Vert \Psi_y(x^*) \right\Vert_p  \leq C \cdot
\left(\int_{B_{F^{**}}} \vert \varphi(x^*) \vert^{p}
d\mu(\varphi) \right)^{\frac{1}{p}},
\end{equation*}
for all $x^*\in B_{F^{*}}$, i.e.,
\begin{equation*}
\left( \sum_{j=1}^{\infty}\vert x^*(T(x_j)) \vert ^{p}\right) ^{\frac{1
}{p}} \leq C \cdot \left(\int_{B_{F^{**}}} \vert
\varphi(x^*) \vert^{p} d\mu(\varphi) \right)^{\frac{1}{p}},
\end{equation*}
for all $x^*\in B_{F^{*}}$.

\noindent b) $\Rightarrow$ a) Still considering $y=(T(x_{j}))_{j=1}^\infty$, if $(x_{j})_{j=1}^\infty \in \ell_p^w(E)$, by (\ref{777}) we obtain that
$(x^*(T(x_{j})))_{j=1}^\infty \in \ell_{p}$, for all $x^* \in
B_{F^{'}}$. Thus $\Psi_y:F^{*} \rightarrow \ell_{p}$ is $p$-summing and, using Theorem \ref{caracter} again, we conclude that $(T(x_{j}))_{j=1}^\infty \in \ell_p^{\text{mid}}(F)$.

\noindent b) $\Leftrightarrow$ c) It is immediate.
\end{proof}

\begin{corollary}
	If $1\le p\le q<\infty$, then $	W_{p}^{\mathrm{mid}} \subseteq
	W_{q}^{\mathrm{mid}}$.
\end{corollary}

\begin{proof} Let $E$ and $F$ be Banach spaces and $T \in W_{p}^{\mathrm{mid}}(E;F)$. Using a) $\Rightarrow$ c) of the Theorem \ref{teopietsch}, there are a constant $C>0$ and	a regular Borel probability measure $\mu$ on $B_{F^{^{** }}}$ such that, for every $k\in \mathbb{N}$,
	\begin{equation*}  
	\left( \sum_{j=1}^{k}\vert x^*(T(x_j)) \vert ^{p}\right) ^{\frac{1 }{p}%
	} \leq C \cdot \left(\int_{B_{F^{^{** }}}} \vert \varphi(x^*)
	\vert^{p} d\mu(\varphi) \right)^{\frac{1}{p}},
	\end{equation*}
	for all $x^*\in B_{F^{^{* }}}$ and all $x_{1},..., x_{k} \in E$. So, for $1\le p\leq q<\infty$ and using the canonical inclusions between $L_{p}$ spaces and between $\ell_{p}$ spaces, we obtain, for every $k\in \mathbb{N}$,
\begin{align*}
\left( \sum_{j=1}^{k}\vert x^*(T(x_j)) \vert ^{q}\right) ^{\frac{1 }{q}%
} & \le \left( \sum_{j=1}^{k}\vert x^*(T(x_j)) \vert ^{p}\right) ^{\frac{1 }{p}%
} \\
& \leq C \cdot \left(\int_{B_{F^{^{** }}}} \vert \varphi(x^*)
\vert^{p} d\mu(\varphi) \right)^{\frac{1}{p}}  \\
& \leq C \cdot \left(\int_{B_{F^{^{** }}}} \vert \varphi(x^*)
\vert^{q} d\mu(\varphi) \right)^{\frac{1}{q}}, 
\end{align*}
for all $x^*\in B_{F^{^{* }}}$ and all $x_{1},..., x_{k} \in E$. The conclusion follows using c) $\Rightarrow$ a) from Theorem \ref{teopietsch}.
\end{proof}

\bigskip

\noindent Departamento de Ci\^{e}ncias Exatas\newline
Universidade Federal da Para\'iba\newline
58.297-000 -- Rio Tinto -- Brazil\newline
and\newline
Departamento de Matem\'atica\newline
Universidade Federal da Para\'iba\newline
58.051-900 -- Jo\~ao Pessoa -- Brazil\newline
e-mails: jamilson@dcx.ufpb.br and jamilsonrc@gmail.com\newline

\noindent Departamento de Matem\'atica\newline
Universidade Federal da Para\'iba\newline
58.051-900 -- Jo\~ao Pessoa -- Brazil\newline
e-mail: joedson@mat.ufpb.br and joedsonmat@gmail.com

\end{document}